\documentclass{article}

\usepackage{amsmath}
\usepackage{amsfonts}
\usepackage{graphicx}

\reversemarginpar

	\usepackage[utf8]{inputenc}
	\usepackage{setspace}
	\usepackage{amsmath,amsthm, amssymb}
	\usepackage[normalem]{ulem} 

\newtheorem{theorem}{Theorem}
\newtheorem{corollary}{Corollary}

\theoremstyle{definition}

\newtheorem*{remark}{Remark}
\newtheorem{example}{Example}

\begin{document}

\title{Cubic Polynomials, Linear Shifts, and Ramanujan Cubics.}
\author{Gregory Dresden,  Prakriti Panthi, 
Anukriti Shrestha, Jiahao Zhang}

\maketitle

\begin{abstract}
We show that every monic  polynomial of degree three
with complex coefficients and no repeated roots is either a
(vertical and horizontal) translation of $y=x^3$ or 
can be composed with a linear function to obtain
a Ramanujan cubic. As a result, we gain some new insights 
into the roots  of cubic polynomials. 
\end{abstract}

\section{Introduction.}\label{SectionIntroduction}
It's hard to pick out a favorite from Ramanujan's nearly-uncountable collection
of delightful identities, but these two have to be near the top of anyone's list:
\begin{align}
\sqrt[3]{1/9} \ - \ \sqrt[3]{2/9} \ +\  \sqrt[3]{4/9}\  &= \ \sqrt[3]{\sqrt[3]{2} - 1}\label{Ram1}\\[1.2ex]
\sqrt[3]{\cos \frac{2\pi}9} \ + \ \sqrt[3]{\cos \frac{4\pi}9}
\ -\  \sqrt[3]{\cos \frac{\pi}9}
\ &= \ 
\sqrt[3]{\frac{3}{2}\left( \sqrt[3]{9} - 2\right)}\label{Ram2}
\end{align}
Both of these equations appear in 
Ramanujan's notebooks \cite{Ber},
and they have been studied in a number
of papers. Landau \cite{Lan1, Lan2} treated the first equation
as an example of ``nested radicals", while Berndt and Bhargava \cite{BB} gave an proof of the second 
equation using only elementary methods. 
It turns out that both equations are related to the roots
of a special class of degree-3 polynomials. 
In keeping with past work, we will define a {\em Ramanujan simple cubic} (RSC) to be a polynomial 
with (possibly complex) coefficients
of the form
\[
\textstyle p_B(x) = x^3 -\left(\frac{3+B}{2}\right) x^2 - \left(\frac{3-B}{2}\right) x + 1.
\]
We will prove that almost every cubic is just 
a linear shift away from a Ramanujan simple cubic, and this
will allow us to recapture the two formulas above, and also to  come up with some lovely new identities,
such as the deceptively simple formula
\begin{equation}
 2\sqrt{6}\cos \frac{11\pi}{36} \ +\ 6\cos \frac{10\pi}{36} \ = \ 
 \left(3\sqrt{2}+\sqrt{6}\right)\cos \frac{\pi}{36} \label{cos36}
\end{equation}
and the rather surprising fact that 
\begin{equation}
\frac{1}{2}\left(-5+\sqrt{13}+2\sqrt{26-6\sqrt{13}} \cos \frac{\pi}{26}\right) \label{cos26}
\end{equation}
is a solution to $x^3 + x^2 -4x + 1 = 0$

In section \ref{SectionRSC} we discuss the properties of
these RSC polynomials, and in 
section \ref{SectionMainResult} we prove our main result. 
This will lead to many nice examples in section \ref{SectionExamples}.


\section{Ramanujan Simple Cubics.}\label{SectionRSC}

Surprisingly, these RSC's have been studied in one form or another for over a hundred years. In 1911, Dickson \cite{Dic} discussed
integral solutions to $p_B(x) = 0$ modulo a prime.
More recently, a number of authors \cite{BCMA, She, Wit} have 
studied a slightly more general class of polynomials
they call {\em Ramanujan cubics}, which are simply our
RSC polynomials $p_B(x)$ but with $x$ replaced by $x/s$.
Similarly, 
if we replace the $x$ in $p_B(x)$ with $-x$, we get
the 
{\em Shanks polynomials}, so called because 
they generate what 
Shanks called the  ``simplest cubic fields" \cite{Sha}.
Foster's paper \cite{Fos} has an excellent review of 
earlier work on the Shanks polynomials and the simplest 
cubic fields; he also proved that every degree-three cyclic
extension of the rationals is generated by a Shanks polynomial (which implies the same for our RSC); this was done earlier by Kersten and Michali\v{c}ek \cite{KM}.
Also, Lehmer \cite{Leh} and Lazarus \cite{Laz} have shown that the minimal polynomials for so-called 
{\em cubic Gaussian periods}, when composed with 
some $x-a$ for $a$ an integer, will equal one of the 
Shanks polynomials (and thus are related to our RSC's). 

The following theorem illustrates some of  
the remarkable properties of Ramanujan simple
cubics (RSC).

\begin{theorem}\label{theorem1}
For $ p_B(x) = x^3 -\left(\frac{3+B}{2}\right) x^2 - \left(\frac{3-B}{2}\right) x + 1$ the Ramanujan simple cubic defined earlier,
\begin{enumerate} 
\item   The roots $r_1, r_2, r_3$ of $p_B(x)$ are always permuted by the order-three map 
$n(x) = \frac{1}{1-x}$. 
\item The roots $r_1, r_2, r_3$
satisfy
\begin{equation}\label{theorem1cuberoots}
\textstyle \sqrt[3]{r_1} \ + \ \sqrt[3]{r_2} \ +\  \sqrt[3]{r_3} \ = \  \sqrt[3]{\left(\frac{3+B}{2}\right)
\ - \ 6 \ + \  3\sqrt[3]{\frac{27 + B^2}{4}}}
\end{equation}
so long as, for complex arguments, we choose the appropriate values for the 
cube roots. 
\item  If we define the elements of the set $\{s_1, s_2, \dots, s_6\}$ as 
\begin{equation}
\textstyle 
s_k \ = \ \frac{1}{3}
\left( \left(\frac{3+B}{2}\right)\  + \ \sqrt{27+B^2} \ 
\cos\left( \frac{k \pi}{3} +  \frac{1}{3} \arctan \frac{3\sqrt{3}}{B} \right) \right)  \label{arctan}
\end{equation}
then for $B\geq 0$ the roots of $p_B(x)$ are $\{s_2, s_4, s_6\}$ and
for $B\leq 0$ the roots of $p_B(x)$ are $\{s_1, s_3, s_5\}$.


\end{enumerate} 

\end{theorem}

\begin{remark}
Although equation (\ref{arctan}) is not actually
defined at $B=0$, we can interpret it at that value by
simply taking the limit of (\ref{arctan}) as $B$ approaches $0$. 
Surprisingly, whether we 
have $B$ approach $0$ from above or from below, we end up with the same answer: the three values of 
$\{s_2, s_4, s_6\}$ and 
the three values of 
$\{s_1, s_3, s_5\}$
coincide at 
$\{-1, 1/2, 2\}$, which are indeed the three roots of
$p_0(x) = x^3 - 3/2\, x^2 - 3/2\, x + 1$.
\end{remark}
\begin{proof}
For part 1, a quick calculation
gives us that 
$-p_B\left(\frac{1}{1-x}\right) \cdot (1-x)^3 = p_B(x)$.
Since  $1$ is never a root of $p_B(x)$, this shows that if
$r_1$ is a root of $p_B$ then so also is $\frac{1}{1-r_1}$.
This is enough to show that 
$n(x)$ permutes the roots so long as 
$\frac{1}{1-r_1}$ is different from $r_1$; 
in the case where 
$\frac{1}{1-r_1}$ equals $r_1$,
this implies that $r_1$ is a primitive 
sixth root of unity, which means 
$B = \pm i \sqrt{27}$ and all three roots 
of $p_B(x)$ are identical (and hence, technically, are
still ``permuted" by $\frac{1}{1-x}$). 

For part 2, there is an elementary proof in \cite[p.~652]{BB} 
of  a nearly identical statement for the roots of the 
Shanks polynomial $x^3 - ax^2 - (a+3)x - 1$; the
roots of this Shanks polynomial are the negatives of
the roots of the Ramanujan polynomial 
$x^3 + a x^2 - (a+3)x + 1$ with $B = -2a-3$ and so the identity follows. This proof
also appears in  \cite[p.~22]{Ber}.

For part 3, we refer the reader to the similar proof 
for Shanks polynomials in \cite[Theorem 7]{BCMA}; another version of this formula (without proof, and for
just one root) can be found 
in \cite{Lou}.
\end{proof}

\begin{example}
We can now easily show that equations (\ref{Ram1}) and (\ref{Ram2}) arise from equation (\ref{theorem1cuberoots}) of Theorem \ref{theorem1}.
For equation (\ref{Ram1}), we take $p_B(x)$ with $B=0$ 
which has roots $1/2, -1, 2$, and so 
equation (\ref{theorem1cuberoots})
gives us
\[
\textstyle \sqrt[3]{1/2} \ + \ \sqrt[3]{-1} \ +\  \sqrt[3]{2} \ = \  \sqrt[3]{\left(\frac{3}{2}\right)
\ - \ 6 \ + \  3\sqrt[3]{\frac{27}{4}}}
\]
and after multiplying through by $\sqrt[3]{2/9}$ and doing 
some simplifying on the right, we get the desired equation.

As for (\ref{Ram2}), we note that 
the minimal polynomial for 
$2 \cos 2\pi/9$ is $x^3 - 3x + 1$, a Ramanujan simple cubic
with $B=-3$. It's easy to show that the other two
roots are $2 \cos 4\pi/9$ and $-2 \cos \pi/9$, and so 
equation (\ref{theorem1cuberoots})
gives us
\[
\textstyle \sqrt[3]{2 \cos 2\pi/9}
\ + \ \sqrt[3]{2 \cos 4\pi/9} \ +\  \sqrt[3]{-2 \cos \pi/9} \ = \  \sqrt[3]{\left(\frac{3-3}{2}\right)
\ - \ 6 \ + \  3\sqrt[3]{\frac{27+9}{4}}}
\]
and after simplifying the right and dividing by
$\sqrt[3]{2}$ we obtain the desired formula.
\end{example}

In the previous example, we began with a particular Ramanujan simple cubic and then derived statements about its roots. We can reverse the process, as seen next. 

\begin{example}
Suppose we wish to create a Ramanujan simple cubic with $x_1 = \sqrt{3}-1$ as one of its roots. We know that the other two roots must satisfy $x_2 = n(x_1)$ and 
$x_3 = n(x_2)$, where $n(x) = \frac{1}{1-x}$. This leads to 
$x_2 = 2+\sqrt{3}$ and $x_3 = (1 - \sqrt{3})/2$, and 
the polynomial $(x-x_1)(x-x_2)(x-x_3)$ is easily calculated to be a Ramanujan simple cubic
with $B = 3\sqrt{3}$. This leads to  
a particularly nice formulation of equation (\ref{theorem1cuberoots}); 
after some simplification (and after multiplying through by $\sqrt[3]{2}$ on
both sides) we obtain the following unexpected equation:
\[
\sqrt[3]{2\sqrt{3}-2} \ - \  \sqrt[3]{\sqrt{3} - 1 }\ + \ \sqrt[3]{2\sqrt{3}+4} 
\ = \ 
\sqrt{3}\cdot\sqrt[3]{1 + \sqrt{3}\left(\sqrt[3]{4}-1\right)}.
\]
\end{example}

\section{Main Result.}\label{SectionMainResult}
For $f(x) = x^3 + P x^2 + Q x + R$ a polynomial with 
(possibly) complex coefficients, we
note that its discriminant is 
\begin{align*}
\Delta &= P^2 Q^2 - 4 Q^3 - 4 P^3 R + 18 P Q R - 27 R^2,
\intertext{and we recall that a polynomial has no repeated roots if and only if its 
discriminant $\Delta$ is not zero. With this in mind, we  define the following 
two values (taken from their original definitions in \cite[p.~468]{Ser2}):}
a &= \frac{  \sqrt{\Delta} -  (9R-PQ) }{2\sqrt{\Delta}}, \\[1.2ex]
c &= \frac{6Q - 2P^2}{2\sqrt{\Delta}}.
\end{align*} 
We now state our main result. 

\begin{theorem}\label{theorem2}
Let $f(x) = x^3 + P x^2 + Q x + R$ have
non-repeated roots $t_1, t_2, t_3$,
and let 
$a$ and $c$ be as defined above.
\begin{enumerate}
\item If $c=0$, then there exists $h$ and $k$ such that $f(x) = (x-h)^3 + k$. 
In other words, $f(x)$ is a translation of $x^3$ (by $h$ units horizontally and
$k$ units vertically). 
\item If $c \not=0$, then $f\left(\frac{a-x}{c}\right) \cdot (-c)^3$ equals the 
Ramanujan simple cubic 
$p_B(x) = x^3 -\left(\frac{3+B}{2}\right) x^2 - \left(\frac{3-B}{2}\right) x + 1$, 
with $B =  6a + 2cP -  3$. In particular,
the set of roots of $p_B(x)$ are $\{a - c\cdot t_1,\ 
a - c \cdot t_2, \ a - c \cdot t_3\}$.
\end{enumerate}
\end{theorem}

\begin{proof}
First, suppose $c=0$. Then $Q = P^2/3$ and so 
$f(x)$ can be written as $(x - h)^3 + k$ with 
$h = -P/3$ and $k = R - P^3/3$. 

Next, for $c \not= 0$, it is possible to use brute force
to show that 
$f\left(\frac{a-x}{c}\right) \cdot (-c)^3$ equals 
$p_B(x)$, but that does not provide much insight into
the problem. Instead, we offer the following more detailed
explanation. The key can be found in Serret's classic
algebra textbook \cite[p.~468]{Ser2} 
from the mid nineteenth century. 
In pursuit of an entirely unrelated problem, 
Serret defined the $a$ and $c$ seen above, along
with the following: 
\begin{align*} 
b &= \frac{2Q^2-6PR}{2\sqrt{\Delta}} \\[1.2ex]
d &= 1-a
\end{align*}
Serret showed that $m(x) = \frac{a x + b}{c x + d}$ is of
order three under composition,  permutes the roots $t_1, t_2, t_3$ 
of the cubic $f(x) = x^3 +
Px^2 + Qx + R$, and has the property that $ad - bc = 1$.  
Now, we would like to transform the cubic $f(x)$ into a new cubic  whose
roots are permuted by $n(x) = \frac{1}{1-x}$, and one way to do that
is to first find a linear map $q(x)$ such that 
$(q^{-1} \circ m \circ q) (x) = n(x)$, and then to consider 
the composition $(f\circ q)(x)$. This composition 
would have as roots the numbers $q^{-1}(t_1), q^{-1}(t_2), q^{-1}(t_3)$, and 
furthermore these roots would be permuted by 
$(q^{-1} \circ m \circ q) (x) = \frac{1}{1-x}$. We can then show this 
composition must be a Ramanujan simple cubic. 

With this in mind, it remains to find our $q(x)$ such that 
$(q^{-1} \circ m \circ q) (x) = n(x)$.
This is a fairly easy task if one 
uses the language of  M\"obius transforms (see, for example, \cite{Dre1}).
Since $n(x)$ takes $\infty$ to $0$ to $1$ back to $\infty$, and 
$m(x)$ takes $\infty$ to $a/c$ to $-d/c$ back to $\infty$, we can choose $q(x)$ 
to take $\infty$ to $\infty$, and $0$ to $a/c$, and $1$ to $-d/c$.
This gives us $q(x) = \frac{a-x}{c}$ and $q^{-1}(x) = a - cx$. 
We can verify that indeed $(q^{-1} \circ m \circ q) (x) = n(x)$, and that
$f(q(x)) = f\left( \frac{a-x}{c}\right)$ has roots
$a-c\,t_1$, $a-c\,t_2$, and  $a - c\,t_3$ as desired. We can then revert to 
brute force to verify that $f\left( \frac{a-x}{c}\right)$ has the desired 
form of a Ramanujan simple cubic. 
\end{proof}

We can now combine Theorem \ref{theorem2} with Theorem \ref{theorem1}
to give us the following results. 
\begin{corollary}\label{cor1}
Let $f(x) = x^3 + P x^2 + Q x + R$ have
non-repeated 
roots $t_1, t_2, t_3$, 
and let $a$, $B$, and $c$ be as defined in Theorem \ref{theorem2}, with $c \not= 0$.
Then, 
\begin{enumerate} 
\item  The order-three map $n(x) = \frac{1}{1-x}$ permutes the set
$\{a - c\cdot t_1,\ a - c \cdot t_2, \ a - c \cdot t_3\}$.
\item  We have the Ramanujan-style equation
\begin{equation}\label{cor1cuberoots}
\textstyle 
\sqrt[3]{a - c \cdot t_1} \ + \ \sqrt[3]{a - c \cdot t_2} \ +\  \sqrt[3]{a - c \cdot t_3} \ = \  
\sqrt[3]{\left(\frac{3+B}{2}\right) \ - \ 6 \ + \ 
3\sqrt[3]{\frac{27 + B^2}{4}}},
\end{equation}
so long as, for complex arguments, we choose the appropriate values for the cube roots.
\item  If we define the elements of the set $\{u_1, u_2, \dots, u_6\}$ as 
\begin{equation}
\textstyle 
u_k \ = \ \frac{a}{c}  -  \frac{1}{3c}
\left( \left(\frac{3+B}{2}\right)\  + \ \sqrt{27+B^2} \ 
\cos\left( \frac{k \pi}{3} +  \frac{1}{3} \arctan \frac{3\sqrt{3}}{B} \right) \right)  \label{arctan2}
\end{equation}
then for $B\geq 0$ the roots of $f(x)$ are $\{u_2, u_4, u_6\}$ and
for $B\leq 0$ the roots of $f(x)$ are $\{u_1, u_3, u_5\}$.
\end{enumerate} 
\end{corollary}

We note that a similar version of formula (\ref{cor1cuberoots}) 
was presented (without proof) by forum user Tito Piezas III on 
{\tt math.stackexchange.com}.



\section{Examples.}\label{SectionExamples}

\begin{example} Here's a rather
lovely formula which we believe has not been seen before:
\[
\sqrt[3]{3-\sqrt{21} +  8 \cos\frac{2\pi}{21}} \ \  + \ \ 
\sqrt[3]{3-\sqrt{21} +  8  \cos\frac{8\pi}{21}} \ \ + \ \ 
\sqrt[3]{3-\sqrt{21} +  8 \cos\frac{10\pi}{21}} 
=  \sqrt[3]{-1-\sqrt{21} + 6\sqrt[3]{28 - 4 \sqrt{21}}}.
\]
To obtain this, we begin with 
$x^6-x^5-6 x^4+6 x^3+8 x^2-8 x+1$, the minimal polynomial
for $t_1 = 2 \cos 2\pi/21$. This factors in
$\mathbb{Q}(\sqrt{21})$ as two cubics,
and we choose the one which still has $2 \cos 2\pi/21$
as a root. This cubic is 
$x^3+\frac{1}{2} \left(-1-\sqrt{21}\right) x^2+\frac{1}{2} \left(\sqrt{21}-1\right) x+\frac{1}{2} \left(\sqrt{21}-5\right)$, and its other two
roots are 
$t_2 = 2 \cos 8\pi/21$ and $t_3 = 2 \cos 10\pi/21$,
and after doing the computations in Theorem \ref{theorem2}
we obtain $a = \frac{1}{2} \left(3-\sqrt{21}\right)$,
 $c= -2$, and $B = 8 - \sqrt{21}$. 
We then plug these values into formula (\ref{cor1cuberoots}),
multiply through by $\sqrt[3]{2}$,
and apply a few simplifications to obtain the above expression.
\end{example}

\begin{example}\label{ex4}
We can do similar calculations for 
$ 2\cos \pi/18$. This has a minimal polynomial of degree 6, but it factors
in $\mathbb{Q}(\sqrt{3})[x]$ and we choose the degree-three factor  
$g(x)=x^3-3x-\sqrt{3}$. One root of $g(x)$ is indeed $2 \cos \pi/18$, and
the
other two roots  are $2\cos 11\pi/{18}$ and $2\cos {13\pi}/{18}$. Calculating $a,c$ 
as defined in Section \ref{SectionMainResult}, we get $a=2$ and $c=-\sqrt{3}$. Thus, 
by Theorem \ref{theorem2}, we have $\displaystyle g(\frac{a-x}{c}) \cdot 
(-c)^3=x^3-6x^2+3x+1$ which is a particularly nice Ramanujan simple cubic with $B=9$.
(We will return to this cubic in Example \ref{ex4return}.) By Corollary 
\ref{cor1}, we get a nice identity:
\begin{equation}\label{e3}
\sqrt[3]{2 + 2 \sqrt{3} \cos\frac{\pi}{18}}  +
\sqrt[3]{2 + 2 \sqrt{3} \cos\frac{11\pi}{18}} + 
\sqrt[3]{2 + 2 \sqrt{3} \cos\frac{13\pi}{18}} =\sqrt[3]{9}.
\end{equation}
Furthermore, by Theorem \ref{theorem1}, we know the roots of $x^3-6x^2+3x+1$ 
 are permuted by $1/(1-x)$. Therefore, by choosing our roots carefully, we get 
\[
    2+2\sqrt{3}\cos \frac{13\pi}{18} =\frac{1}{1- \Big( 2+2\sqrt{3} \cos \displaystyle \frac{\pi}{18} \Big)}
\]
and this simplifies to 
\[
    2\cos \frac{\pi}{18}+\cos \frac{13\pi}{18}+\sqrt{3} \cos \frac{14\pi}{18} =0
\]
which reduces to 
\begin{equation}
   \cos \frac{5\pi}{18} \ =\   2\cos \frac{\pi}{18}  -\sqrt{3} \cos \frac{4\pi}{18}.
   \label{cos18}
\end{equation}
\end{example}

\begin{example}
In an effort to find more equations like (\ref{cos18}), 
we look at the minimal polynomials for 
$2 \cos \pi/36$ and $2 \cos \pi/42$. 
Both have minimal polynomials of degree 12, and both can be factored
down into degree three polynomials by adjoining appropriate square roots
to the rationals. 
By following the same steps as in the previous example we can arrive at
the following two identities:
\begin{align}
 2\sqrt{6}\cos \frac{11\pi}{36} \ +\ 6\cos \frac{10\pi}{36} \ - \ 
 \left(3\sqrt{2}+\sqrt{6}\right)\cos \frac{\pi}{36}  &= 0 \label{cc1} \\[1.2ex]
(\sqrt{3}-\sqrt{7})\cos \frac{\pi}{42} \ -\ 2\sqrt{7}\cos \frac{25\pi}{42} \ -\ 8\cos \frac{\pi}{42}\cos \frac{25\pi}{42} \ &=\ 3.  \label{cc2}
\end{align}
It's probably just a coincidence, but 
$\left(3\sqrt{2}+\sqrt{6}\right)\cos \pi/{36}$ from formula (\ref{cc1}) is almost identical (to six decimal places) to  $20/3$. Also, note that (\ref{cc1}) is equation (\ref{cos36}) from the beginning of the article. 
\end{example}

\begin{example}\label{ex4return}
 Returning our attention to Example \ref{ex4}, we note that 
Theorem \ref{theorem2} gave us the particularly nice
cubic $x^3-6x^2+3x+1$ and gave us that one of its roots is 
$2 + 2\sqrt{3} \cos \pi/18$. Likewise, if we begin with 
$2 \cos \pi/26$, we can factor its minimal (degree-12) polynomial down to
a degree-3 polynomial with irrational coefficients, apply Theorem \ref{theorem2},
and end up with another particularly nice polynomial, this time
$x^3 + x^2 - 4x + 1$, one of whose roots is given in equation (\ref{cos26}) at the
beginning of this paper. 

It turns out that, as seen in \cite{Leh}, these two polynomials are also just
an integer shift from the minimal polynomials for certain
{\em cubic Gaussian periods}. The exact nature of these objects is beyond the scope of this article; for our purposes, we can consider them to be sums of 
roots of unity with their inverses. Suffice it to say that
this recognition leads us to discover that 
$x^3-6x^2+3x+1$ is the minimal polynomial for the following 
three numbers:
\[
\left\{ 2 + 2 \cos \frac{\pi}{9} + 2 \cos \frac{2\pi}{9},\ \ \ 
 2 + 2 \cos \frac{4\pi}{9} + 2 \cos \frac{7\pi}{9}, \ \ \ 
2 +   2 \cos \frac{5\pi}{9} + 2 \cos \frac{8\pi}{9} 
\right\}.
\]  
Comparing these with $2 + 2\sqrt{3} \cos \pi/18$ leads us to the identity
\[
2 + 2\sqrt{3} \cos \frac{\pi}{18}  \ = \ 2 + 2 \cos \frac{\pi}{9} + 2 \cos \frac{2\pi}{9}.
\]
Unfortunately, this simplifies to a triviality. However, 
along these lines, we also discover that $x^3 + x^2 - 4x + 1$
is the minimal polynomial for the following 
three numbers:
\[
\left\{    2 \cos \frac{2\pi}{13} + 2 \cos \frac{10\pi}{13}       ,\ \ \ 
      2 \cos \frac{4\pi}{13} + 2 \cos \frac{6\pi}{13}          , \ \ \ 
        2 \cos \frac{8\pi}{13} + 2 \cos \frac{12\pi}{13} 
\right\}.
\]
After comparing to the solution in equation (\ref{cos26}), we obtain this (non-trivial)
identity,
\[
-5\ +\ \sqrt{13}\ +\ 2\sqrt{26-6\sqrt{13}} \, \cos \frac{\pi}{26}
\ =\ 4 \cos \frac{4\pi}{13} \ +\  4 \cos \frac{6\pi}{13},
\]
and this really is a lovely formula.

\end{example}

\begin{example}
We finish with an example that does not involve cosines. 
Consider the polynomial 
$f(x) = (x-1)(x-\sqrt{2})(x+\sqrt{3})$. This is not Ramanujan, but 
when we apply the methods of Theorem \ref{theorem2} we obtain
a Ramanujan polynomial $p_B(x)$ with 
$B= -6-\sqrt{2}- 5\sqrt{3}+\sqrt{6}$, and one of its roots
is $-1-\sqrt{2}-\sqrt{3}-\sqrt{6}$. After trying various values of 
$k$ with 
formula (\ref{arctan}), we find that 
\[
-1-\sqrt{2}-\sqrt{3}-\sqrt{6}
=
\ \frac{1}{3}
\left( \left(\frac{3+B}{2}\right)\  + \ \sqrt{27+B^2} \ 
\cos\left( -\pi +  \frac{1}{3} \arctan \frac{3\sqrt{3}}{B} \right) \right)  
\]
and after applying our value of $B$ and simplifying, we obtain the 
following formula:
\[
3+5 \sqrt{2}+\sqrt{3}+7 \sqrt{6}
=
2 \sqrt{2 \left(73-9 \sqrt{2}+28 \sqrt{3}-\sqrt{6}\right)} \cos \left(\frac{1}{3}
   \tan ^{-1}\left(\frac{3 \sqrt{3}}{-6-\sqrt{2}-5
   \sqrt{3}+\sqrt{6}}\right)\right),
\]
and this is surprising if for no other reason than the relatively small size
of the coefficients.
\end{example}

\

\bibliography{RSC}

\begin{thebibliography}{10}

\bibitem{BCMA}
Stefano Barbero, Umberto Cerruti, Nadir Murru, and Marco Abrate.
\newblock Identities involving zeros of {R}amanujan and {S}hanks cubic
  polynomials.
\newblock {\em J. Integer Seq.}, 16(8):Article 13.8.1, 13, 2013.

\bibitem{Ber}
Bruce~C. Berndt.
\newblock {\em Ramanujan's notebooks. {P}art {IV}}.
\newblock Springer-Verlag, New York, 1994.

\bibitem{BB}
Bruce~C. Berndt and S.~Bhargava.
\newblock Ramanujan---for lowbrows.
\newblock {\em Amer. Math. Monthly}, 100(7):644--656, 1993.

\bibitem{Dic}
L.~E. Dickson.
\newblock Note on cubic equations and congruences.
\newblock {\em Ann. of Math. (2)}, 12(3):149--152, 1911.

\bibitem{Dre1}
Gregory~P. Dresden.
\newblock There {A}re {O}nly {N}ine {F}inite {G}roups of {L}inear {F}ractional
  {T}ransformas with {I}nteger {C}oefficients.
\newblock {\em Math. Mag.}, 77(3):211--218, 2004.

\bibitem{Fos}
Kurt Foster.
\newblock H{T}90 and ``simplest'' number fields.
\newblock {\em Illinois J. Math.}, 55(4):1621--1655 (2013), 2011.

\bibitem{KM}
Ina Kersten and Johannes Michali\v~cek.
\newblock A characterization of {G}alois field extensions of degree {$3$}.
\newblock {\em Comm. Algebra}, 15(5):927--933, 1987.

\bibitem{Lan1}
Susan Landau.
\newblock Simplification of nested radicals.
\newblock {\em SIAM J. Comput.}, 21(1):85--110, 1992.

\bibitem{Lan2}
Susan Landau.
\newblock How to tangle with a nested radical.
\newblock {\em Math. Intelligencer}, 16(2):49--55, 1994.

\bibitem{Laz}
Andrew~J. Lazarus.
\newblock Gaussian periods and units in certain cyclic fields.
\newblock {\em Proc. Amer. Math. Soc.}, 115(4):961--968, 1992.

\bibitem{Leh}
Emma Lehmer.
\newblock Connection between {G}aussian periods and cyclic units.
\newblock {\em Math. Comp.}, 50(182):535--541, 1988.

\bibitem{Lou}
St\'ephane Louboutin.
\newblock The exponent three class group problem for some real cyclic cubic
  number fields.
\newblock {\em Proc. Amer. Math. Soc.}, 130(2):353--361, 2002.

\bibitem{Ser2}
Joseph-Alfred Serret.
\newblock {\em Cours d'alg\`ebre sup\'erieure. {T}ome {II}}.
\newblock Les Grands Classiques Gauthier-Villars. [Gauthier-Villars Great
  Classics]. \'Editions Jacques Gabay, Sceaux, 1992.
\newblock Reprint of the fourth (1879) edition.

\bibitem{Sha}
Daniel Shanks.
\newblock The simplest cubic fields.
\newblock {\em Math. Comp.}, 28:1137--1152, 1974.

\bibitem{She}
Vladimir Shevelev.
\newblock On {R}amanujan cubic polynomials.
\newblock {\em South East Asian J. Math. Math. Sci.}, 8(1):113--122, 2009.

\bibitem{Wit}
Roman Witu\l{a}.
\newblock Full description of {R}amanujan cubic polynomials.
\newblock {\em J. Integer Seq.}, 13(5):Article 10.5.7, 8, 2010.

\end{thebibliography}
\bibliographystyle{plain}

\ 

\ 


\end{document}